\documentclass[12pt]{article}
\usepackage{amssymb,xcolor,epic}
\usepackage{amsmath}
\usepackage{graphicx}
\usepackage{enumerate}
\usepackage{mathrsfs}
\usepackage{latexsym}
\usepackage{psfrag}
\usepackage{mathrsfs}
\usepackage{multicol}
\usepackage{algorithm}
\usepackage{algpseudocode}
\usepackage{pgf,tikz,pgfplots}
\usepackage{mathrsfs}
\usepackage{tikz}
\usetikzlibrary{shapes.geometric}
\usetikzlibrary{arrows.meta,arrows}
\usepackage{comment}

\newcommand{\qed}{\hfill $\square$}

\newcommand{\remark}{\noindent{\bf Remark.} }

\newenvironment{proof}{\noindent{\em Proof}.}{\qed\bigskip}

\newtheorem{theorem}{Theorem}[section]
\newtheorem{lemma}[theorem]{Lemma}

\newtheorem{corollary}[theorem]{Corollary}

\newtheorem{conjecture}{Conjecture}[section]

\usepackage{comment}
\newcommand{\comments}[1]{}

\title{Every connected graph admits a local antimagic orientation and almost every graph admits an antimagic orientation}

\author{
Eranda Dhananjaya
\thanks{Department of Applied Mathematics, National Chung Hsing University, Taichung 40227, Taiwan
{\tt Email:edhananjaya1991@gmail.com} supported by NSTC 112-2115-M-005 -002 -MY2.}
\and
Wei-Tian Li
\thanks{Department of Applied Mathematics, National Chung Hsing University, Taichung 40227, Taiwan
{\tt Email:weitianli@nchu.edu.tw} supported by NSTC 112-2115-M-005 -002 -MY2.}
}

\date{\today}

\begin{document}

\maketitle

\begin{abstract}
An undirected graph $G$ is said to admit an antimagic orientation if there exist an orientation $D$ and a bijection between $E(G)$ and $\{1,2,\ldots,|E(G)|\}$ such that any two vertices have distinct vertex sums, where the vertex sum of a vertex is the sum of the labels of the in-edges minus that of the out-edges incident to the vertex. 
It is conjectured by Hefetz, M\"{u}tze, and Schwartz that every connected graph admits an antimagic orientation. A weak version of this problem is to require the distinct vertex sums only for the adjacent vertices. In that case, we say the graph admits a local antimagic orientation. 
Chang, Jing, and Wang~\cite{CJW20} conjectured that every connected graph admits a local antimagic orientation. In this paper, we give an affirmative answer to the conjecture of Chang et al., and show that almost every graph satisfies the conjecture of Hefetz et al. 
\end{abstract}


\section{Introduction}

Given an undirected simple graph on $m$ edges, can one label its edges with $1,2,\ldots, m$ one on one so that when every vertex receives the sum of the labels of all edges incident to the vertex, all sums are pairwise distinct?
This graph labeling problem was proposed by by Hartsfield and Ringel~\cite{HR90} in the 1990s. They conjectured that such kind of labelings exist for all connected graphs on at least two edges. 
Nowadays, many graphs have been verified to satisfy the conjecture and many analogous problems are posed.  
Let $G$ be a graph, and let $V(G)$ and $E(G)$ be the vertex and edge sets of $G$, respectively. Whenever there is a mapping from $E(G)$ to a set $L$ of real numbers, we define the {\em vertex sum} of a vertex, induced by the mapping, to be the sum of the labels of the edges incident to the vertex. 
The original problem of Hartsfield and Ringel is looking for a mapping $\tau$ such that  
\begin{description}
    \item{(1)} $\tau$ is an injection.
    \item{(2)} $L=\{1,2\ldots,|E(G)|\}$.
    \item{(3)} the vertex sums are pairwise distinct.
\end{description}
A graph equipped with a mapping $\tau$ satisfying all the conditions is called {\em antimagic}. The statement of the conjecture of Hartsfield and Ringel in~\cite{HR90} is the following:
\begin{conjecture}\cite{HR90}\label{HR90}
Every connected graph other than $K_2$ is antimagic.    
\end{conjecture}
The most pioneering results on Conjecture~\ref{HR90} are the result of Alon, Kaplan, Lev, Roditty, and Yuster   
for dense graphs~\cite{A04}, and the results of Cranston, Liang, and Zhu~\cite{CLZ15} for odd regular graphs, and Chang, Liang, Pan, and Zhu for even regular graphs~\cite{C16}. 
In the following, we introduce some variations of the ordinary antimagic labeling.
For the set $L$ in condition (2), if it is replaced by a set of $|E(G)|$ consecutive integers starting with $k+1$, then the graph with such a mapping is called {\em $k$-shifted antimagic}, see~\cite{CCLP21};   
If for each set $L$ of $|E(G)|$ positive real numbers there is a corresponding mapping, then the graph is {\em universal antimagic}, named by Matamala and Zamora~\cite{MZ17}; 
For the up-to-date results on these different types of antimagic problems, we recommend the readers the survey~\cite{G19}.

Some graph labeling problems are aiming to find distinct vertex sums for adjacent vertices using $L$ as small as possible. 
Let $k$ be a positive integer. 
A {\em $k$-edge-weighting} is a mapping from $E(G)$ to $\{1, 2, \ldots, k\}$ and is {\em neighbor-sum-distinguishing (nsd)} if the vertex sums of any two adjacent vertices are distinct. 
This concept was introduced with the 1-2-3 Conjecture in 2004 by Karo\'{n}ski, \L uczak, and Thomason in~\cite{KLT04}.

\begin{conjecture}~\cite{KLT04}\label{KLT04}
     Every graph with no isolated edge admits an nsd  $3$-edge-weighting.
\end{conjecture}
The above conjecture is generally known as the 1-2-3 Conjecture. 
The most significant progress toward the 1-2-3 Conjecture thus far is attributed to Kalkowski, Karoński, and Pfender~\cite{KKP10}, who established that every graph lacking isolated edges can be equipped with an nsd 5-edge-weighting. 
Inspired by the 1-2-3 Conjecture, Bensmail, Senhaji, and Lyngsie~\cite{BSL17} introduced the concept of local antimagic labelings for graphs. 
A {\em local antimagic labeling} of a graph $G$ is an nsd bijective mapping denoted as $\phi: E(G) \rightarrow \{1, 2, \ldots, |E(G)|\}$. Bensmail et al. in \cite{BSL17} proposed a conjecture asserting that every graph without isolated edges can be assigned a local antimagic labeling. 
Subsequently, Haslegrave~\cite{H18} validated this conjecture through the utilization of the probabilistic method.

The exploration of antimagic labeling in directed graphs began considerably later, with its initiation by Hefetz, M\"{u}tze, and Schwartz in 2010~\cite{HMS10}. 
For a directed graph $\Vec{G}$, the vertex sum of a vertex $v$ is defined as the sum of the labels of all edges entering $v$ minus the sum of the labels of all edges leaving $v$. 
An {\em antimagic orientation} for an undirected graph $G$ is characterized by the existence of an orientation $D$ on its edge set $E(G)$ and a bijective mapping from $E(G)$ to the set $\{1, 2, \ldots , |E(G)|\}$ 
such that the orientation and the bijective mapping must fulfill the requirement that the vertex sums, as previously defined in this paragraph  are distinct for all vertices in the graph.
In their work~\cite{HMS10}, Hefetz et al. posed two thought-provoking questions. The first question delves into whether every orientation of a connected graph, except for the directed cycle and path on tree vertices, demonstrates antimagic properties, while the second question scrutinizes the possibility of affording an antimagic orientation to each connected graph. 
Furthermore, they provided partial results for these inquiries and introduced the conjecture ``Every connected graph admits an antimagic orientation" in connection with the latter question.

\begin{conjecture}{\rm \cite{HMS10}}\label{HMS10}
 Every connected graph admits an antimagic orientation. 
\end{conjecture}
Over the past decade, the above conjecture has garnered significant attention and garnered substantial support in various studies. Despite this, it still stands as an unsolved question. Diverse research communities have independently demonstrated that several graph families, as cited in works such as \cite{GS20, L18, LSW19, SH19, CSH19, SS20, SYZ19, SX17, YCO19, DY19, CZ19}, are consistent with Conjecture~\ref{HMS10}. 

The efforts to extend the 1-2-3 Conjecture to directed graphs have indeed met with success. 
Borowiecki, Grytczuk, and Pil\'{s}niak~\cite{BGP12} established that every simple directed graph can be equipped with an nsd $2$-edge-weighting. 
Chang, Jing, and Wang~\cite{CJW20} drew motivation from the directed version of the 1-2-3 Conjecture to introduce the concept of local antimagic orientation in a similar manner as follows: 
An undirected graph $G$ is said to admit a {\em local antimagic orientation}, if there exist an orientation $D$ on $E(G)$ and a bijective mapping $\tau$ from the set of arcs in $G_D$ to  the set $L=\{1, \ldots, |E(G)|\}$ such that $\tau$ is nsd. 
Moreover, they proved that every connected graph with maximum degree at most $4$ admits a local antimagic orientation by combinatorial nullstellensatz and proposed the following conjecture.  

\begin{conjecture}{\rm \cite{CJW20}}\label{CJW20}
    Every connected graph admits a local antimagic orientation.
\end{conjecture}
By relaxing the size of $L$, Hu, Ouyang, and Wang~\cite{HOW19} proved  that every $d$-degenerate graph $G$ admits a local $\{1, 2, \ldots, |E(G)|+d+2\}$-antimagic orientation. In other words, one can find an orientation $D$ of $G$ and an injective mapping from the set of the arcs $A(G_D)$ to $L=\{1, 2, \ldots, |E(G)|+d+2\}$ such that the adjacent vertices receive distinct vertex sums.

In this paper, we completely resolve Conjecture~\ref{CJW20} by showing a more general result. Imitating the terminologies in the literature, we say that a graph $G$ is {\em universal antimagic orientable} (resp. {\em universal local antimagic orientable})  provided that for every set $L$ of $|E(G)|$ positive numbers, there exist an orientation $D$ and a bijection from the set of directed edges $A(G_D)$ to $L$ that lead to pairwise distinct vertex sums for all  vertices (resp. any pair of adjacent vertices). We mange to show the following theorem: 

\begin{theorem}\label{NT}
    Every graph without isolated vertices is universal local antimagic orientable.
\end{theorem}

Given a positive integer $n$ and a function $0\le p(n)\le 1$. We denote the random variable $G(n,p(n))$ a graph on $n$ vertices and two vertices are adjacent with probability $p$. 
Particularly, when $p(n)=1/2$,  $G(n,1/2)$ is the {\em random graph} which means all graphs on $n$ vertices have the same probability in the sample space. 
We say that {\em almost every graph has a given property} means
that if $P(n)$ is the probability that a random graph on $n$ vertices has the 
property, then $P(n)\rightarrow 1$ as $n \rightarrow \infty$. The second result presented in this paper is the following theorem on the random graph for Conjecture~\ref{HMS10}.

\begin{theorem}\label{random}
    Almost every graph admits an antimagic orientation.
\end{theorem}

In the next section, we will present the proof of Theorem~\ref{NT} by introducing first a modification of a frequently used lemma for the study of the antimagic orientation problems. The result on the random graphs is given in Section 3.

\section{Local antimagic orientations for graphs}

Let $G$ be an undirected graph. Then $d_G(v)$ denote the degree of the vertex $v$ in $G$.
When $G$ endowed with an orientation $D$,  we use $d_D^+(v)$ and $d_D^-(v)$ to denote the out-degree and in-degree of a vertex $v$, respectively. If an edge labeling $\tau$ is also given, then we denote 
the vertex sum of $v$ with $s_{(D,\tau)}(v)$.
By utilizing Euler circuits, it is possible to approximate the vertex sums that are generated by the circuit. 
This method has been used in 
in several recent publications, such as Lemma 2.2 in \cite{SYZ19}, Lemma 2.1 in\cite{YCO19}, and Lemma 7 in~\cite{SS20}, to show different classes of graphs satisfying Conjecture~\ref{HMS10}. 
We apply the idea in these lemmas to obtain Lemma~\ref{L_1}, which is of significant importance in proving our main theorem.

\begin{lemma}\label{L_1}
Let $G$ be a graph on $m$ edges and $L$
be a set of $m$ positive real numbers. 
There exist an orientation $D$ of $G$ and a bijection $\tau: E(G)\rightarrow L$ such that for
any $v \in V (G)$, 
either $s_{(D, \tau)}(v)<0$ or $s_{(D, \tau)}(v)\le \max L$.  
\end{lemma}

\begin{proof}
If $G$ comprises more than one connected component, we can partition $L$ into disjoint subsets. This allows us to handle each component separately with labels falling within their respective subset. Consequently, it suffices demonstrate the lemma for connected graphs.  
Let 
$V_1 = \{v \in V(G) \mid d_G(v) \text{ is odd}\}$. According to the handshaking lemma, $|V_1|=2t$ is an even number. If $V_1=\varnothing$, we define $G^* = G$, otherwise, we define $G^{*}$ by incorporating $t$ new edges into $G$ in such a way that each new edge connects a pair of vertices in $V_1$, and denote this set of new edges as $E^{*}$. Note that $G^{*}$ may be a multiple graph and the degree of every vertex in $G^{*}$ is an even number.

Now, we construct an Euler circuit, denoted as $C$ in $G^*$ according to the following rules: If $G^* = G$, then we have the flexibility to select any vertex as the starting point and construct $C$. However, if $G^* \neq G$, we opt for a vertex in $V_1$ as our initial vertex. Since this chosen vertex is incident to exactly one edge in $E^*$, we designate this edge as the first edge of $C$. Suppose that
\[C = u_1, e_1, u_2, e_2, \ldots , u_{m+t}, e_{m+t}, u_1,\]
where $u_i$’s are the vertices in $V (G)$, allowing for possible repetitions , and $e_1, e_2, \ldots, e_{m+t}$ are
exactly the edges in $E(G^*).$

Let $e_{i_1}, e_{i_2}, \ldots , e_{i_m}$, $i_1<i_2<\cdots<i_m$, be the edges in $E(G)$. By orienting each edge a direction the same as it is in $C$, we obtain an orientation $D$ of $G$ such that every vertex $v$ in the directed graph $D$ satisfies $\left | d^+_D(v)-d^-_D(v) \right |\le 1$. 
Suppose $L=\{a_1, a_2, \ldots, a_{m}\}$ and $a_j$ is increasing in $j$. Define $\tau : E(G) \rightarrow L$ with $\tau (e_{i_j}) = a_{j}$ for $1 \le  j \le  m$.
The orientation $D$ and labeling $\tau$ induce the vertex sums as follows: 
When $G^* = G$, the a vertex sum of a vertex is the sum of an alternating series
\begin{equation}\label{eq1}
s_{(D, \tau)}(v)=
- a_{j_1}+a_{j_2}-a_{j_3}+\cdots+a_{j_{d_G(v)}},    
\end{equation}
 for $v=u_1$, the initial vertex of $C$, or   
\begin{equation}\label{eq2}
s_{(D, \tau)}(v)=
a_{j_1}-a_{j_2}+a_{j_3}-\cdots-a_{j_{d_G(v)}},    
\end{equation}
where $1\le j_1<j_2<\cdots\le j_{d_G(v)}$. It is straightforward to verify that the vertex sum in~(\ref{eq1}) is less than $a_m$ and the vertex sum in~(\ref{eq2}) is negative.  
For $G^* \neq G$, if $\left | d^+_D(v)-d^-_D(v) \right |=0$, 
then the vertex sum $s_{(D, \tau)}(v)$ has a form the same as~(\ref{eq2}), which is negative.  
Consider $ |d^+_D(v)-d^-_D(v)| =1$. If $v=u_1$, by the choice of $e_1$, we have 
\[s_{(D, \tau)}(v)=
a_{j_1}-a_{j_2}+a_{j_3}-\cdots+a_{j_{d_G(v)}},\]
where $1\le j_1<j_2<\cdots<j_{d_G(v)}$. 
So $s_{(D, \tau)}(v)<a_m$. 
If $v\neq u_1$, then $s_{(D, \tau)}(v)$ can be viewed as the sum of a series similar to (\ref{eq2}) but one term $a_{j_k}$ is missing. 
Once the missing term is associated with a plus sign, then $s_{(D, \tau)}(v)$ is negative. Otherwise, the $s_{(D, \tau)}(v)$ is the missing term plus a negative number, which is less than $a_m$. 
%
%
%
\end{proof}

\remark By reversing the direction of every edge in the above orientation $D$, we can obtain another orientation $D'$ of $G$ that gives   
either $s_{(D', \tau)}(v)>0$ or $s_{(D', \tau)}(v)\ge -\max L$ for
any $v \in V (G)$.

To establish the proof of Theorem \ref{NT}, we employ the technique originally utilized by Cranston, Liang, and Zhu~\cite{CLZ15} to establish the antimagic labeling for odd regular graphs. 
Their approach can be summarized as follows: Begin by selecting a vertex $u$ from the graph $G$. 
Then, partition the vertex set $V(G)$ into levels $V_0, V_1, \ldots, V_d$, where $V_i=\{v\mid d(v, u)=i\}$, $d(v,u)$ is the distance of $v$ and $u$, and $d$ represents the maximum distance of a vertex from $v$. Define $G[V_i]$ as the subgraph induced by $V_i$ and $G[V_i, V_{i-1}]$ as the bipartite subgraph induced by the two sets $V_i$ and $V_{i-1}$. 
For every positive $i$ and every vertex $v$ in $V_i$, reserve an edge $e_v$ that is incident to vertices in $V_i$ and $V_{i-1}$. 
Next, create a labeling $f$ by sequentially labeling the edges in $G[V_d]$, $G[V_d, V_{d-1}]$ and $\{e_v:v\in V_d\}$, and so on, up to $G[V_1]$ and $G[V_1, V_0]$, while ensuring that the smallest unused labels are used and adhering to specific additional rules. 
It is worth mentioning that we can accomplish the partition of the set $V(G)$ and the selection of the edges $e_v$ by using the breadth-first search (BFS) to construct a rooted tree for $G$. 
\bigskip

\noindent{\em Proof of Theorem~\ref{NT}.} 
Since vertices in two components of a graph are not adjacent, it suffices to show the theorem for connected graphs. 

Let $G$ be a connected graph, and let $L$ be a set of $|E(G)|$ positive real numbers. 
First pick a vertex $u$ and construct a rooted tree $T$ using the BFS with root $u$. Suppose that the set $V(G)$ is partitioned into sets $V_0, V_1, \ldots, V_d$, where $V_i$ is the set of vertices in the $i$th level of $T$. 
Then partition the set $L$ into subsets $L_0, L_1,\ldots, L_d$ satisfying
\begin{description}
    \item{(1)} $|E(G[V_i])|=|L_i|$  for $1\le i \le d$ and $L_0=L-(L_1\cup L_2\cup\cdots\cup L_d)$;
    \item{(2)}  For $1\le i\le d$, $\min L_0>\max L_i$.
\end{description}

Next, we apply Lemma~\ref{L_1} to 
construct the orientations $D_i$ of $G[V_i]$ and define the corresponding labeling $\tau_i:E(G[V_i])\rightarrow L_i$ for $1\le i\le d$ such that for 
any $v\in V(G_i)$ either $s_{(D_i,\tau_i)}(v)<0$ or $s_{(D_i,\tau_i)}(v)\le \max L_i$ when $i$ is odd, and 
for any $v\in V(G_i)$ either $s_{(D_i,\tau_i)}(v)>0$ or $s_{(D_i,\tau_i)}(v)\ge -\max L_i$ when $i$ is even.
We complete the orientation of $G$ by directing the edges in $G[V_{i},V_{i-1}]$ from $V_i$ to $V_{i-1}$ if $i$ is odd and from $V_{i-1}$ to $V_i$ if $i$ is even.
For each vertex $v\neq u$, we choose the edge incident to $v$ on the path connecting $u$ and $v$ in $T$ as $e_v$. 
For each unlabeled edge not in $T$, we arbitrary label it with an unused label in $L_0$.
Finally, we label the edges $e_v$.
In the following, we assume that $d$ is odd, the case of even $d$ is analogous. 
For each vertex $v$ in $V_d$, $e_v$ is the only unlabeled edge incident to $v$. 
Recall that we have either $s_{(D_d,\tau_d)}(v)<0$ or $s_{(D_d,\tau_d)}(v)\le \max L_d$. 
Moreover, all edges in $G[V_d,V_{d-1}]$ are directed from $V_d$ to $V_{d-1}$. So the partial vertex sum of a vertex $v$, $s(v)$, induced by the labeled edges is either a negative number or at most $\max L_d$. 
We label each $e_v$ with an unused label $\tau(e_v)$ in $L_0$ with the condition that $\tau(e_v)<\tau(e_{v'})$ implies $s(v)\ge s(v')$. 
Since $\min L_0>\max L_d$, not only the vertex sums are all distinct but also all of them are negative. 
The next step is to label $e_v$ for each vertex $v$ in $V_{d-1}$, which is only unlabeled edge incident to $v$ now. 
The partial vertex sum $s(v)$, induced by the labeled edges is either a  positive number or at least $-\max L_{d-1}$. 
This time we label each $e_v$ with an unused label $\tau(e_v)$ in $L_0$ with the condition that $\tau(e_v)<\tau(e_{v'})$ implies $s(v)\le s(v')$. 
Thus, the vertex sums of all vertices in $V_{d-1}$ are distinct and positive. 
We use the above strategies to label the unlabeled edges $e_v$ in $G[V_{d},V_{d-1}], G[V_{d-1},V_{d-2}],\ldots,G[V_1,V_0]$, accordingly.

Based on the BFS construction, two vertices can only be adjacent to each other in the same level or in two consecutive levels. 
We have demonstrated that for any two consecutive levels, the vertex sums of their vertices exhibit distinct signs. 
Also, we have established that for any pair of vertices belonging to the same $V_i$, the vertex sums are distinct. 
As a conclusion, we complete the proof of Theorem~\ref{NT}.
\hfill\qed\bigskip
 
\section{The random graphs}

We begin with a deterministic result on Conjecture~\ref{HMS10}. 
The {\em eccentricity} of a vertex $v$ in a graph $G$ is the maximum distance of a vertex in $G$ to $v$, and the {\em radius} of a graph $G$ is the minimum eccentricity of a vertex in $G$. 
Let $G$ be a graph of radius two. We can construct a rooted tree of $G$, which has three levels by designating the vertex of the minimum eccentricity as the root.   
Thereby, when applying the labeling method in the proof of Theorem~\ref{NT} to $G$, it turns out that either all vertex sums are pairwise distinct or at most one pair of vertices have the same vertex sum. When the latter case happens, one of the two vertices is the root and the other is a vertex which is not adjacent to the root. 
Indeed, we can make all the vertex sums different by modifying the labeling strategies.

\begin{theorem}
Every graph of radius of two is universal antimagic orientable.
\end{theorem}

\begin{proof}
Let $G$ be a graph of radius two, and $u$ be a vertex with the minimum eccentricity. 
First use BFS to construct a rooted tree $T$ with root $u$.
The terms $V_i$, $L_i$, and $e_v$ are defined analogously as in the proof of Theorem~\ref{NT}. 

We orient the edges, label the edges in $G[V_1]$ and $G[V_2]$ with the labels in $L_1$ and $L_2$, and arbitrary label all the unlabeled edges not in $T$ with the labels in $L_0$ as previous. 
Next, for the edges in $T$ we label all except for one edges incident to $u$ with the unused labels in $L_0$. 
Thus, for each vertex in $V_0\cup V_2$, there is exactly one unlabeled edge incident to it.  
Now, we label these edges with the remaining labels according to the partial vertex sums of the vertices incident to them. 
More precisely, let $s(v)$ be the sum of the labels of the labeled edges incident to $v$. We label each $e_v$ (here we abuse the notation $e_u$ to denote the only unlabeled edge incident to $u$) with an unused label $\tau(e_v)$ in $L_0$ subject to the condition that $\tau(e_v)<\tau(e_{v'})$ implies $s(v)\le s(v')$. 
This guarantees that the vertex sums of the vertices in $V_0\cup V_2$ are all distinct and positive. 
However, we might encounter the problem that two vertices in $V_1$ have the same vertex sum. 
Once this happens, the solution is to rearrange the labels of $e_v$ for $v$ in $V_1$. 
Namely, we compare the partial vertex sums of the vertices in $V_1$ contributed by the edges in $G[V_1]$ and in $G[V_2,V_1]$, then reassign the labels to the edges $e_v$ accordingly to make the vertex sums distinct. 
Since this rearrangement does not change the vertex sum of the root $u$, the vertex sums of the vertices in $V_0\cup V_2$ are all positive and distinct, while the vertex sums for vertices in $V_1$ now are all negative and distinct. So the theorem is proved.
\end{proof}

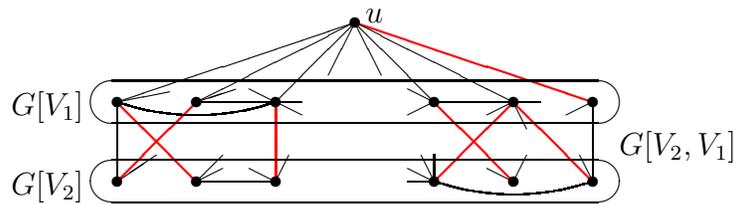
\begin{figure}[ht]
    \centering
\begin{picture}(300,80)
    \textcolor{red}
    {\thicklines
    \put(150,70){\line(3,-1){90}}
    \put(60,10){\line(1,1){30}}
    \put(90,10){\line(-1,1){30}}
    \put(120,10){\line(0,1){30}}
    \put(180,10){\line(1,1){30}}
    \put(210,10){\line(-1,1){30}}
    \put(240,10){\line(-1,1){30}}
    }
    \put(60,10){\circle*{4}}
    \put(90,10){\circle*{4}}
    \put(120,10){\circle*{4}}
    \put(180,10){\circle*{4}}
    \put(210,10){\circle*{4}}
    \put(240,10){\circle*{4}}
    \put(150,10){\oval(200,16)}
    \put(20,5){$G[V_2$]}
    \put(60,40){\circle*{4}}
    \put(90,40){\circle*{4}}
    \put(120,40){\circle*{4}}
    \put(180,40){\circle*{4}}
    \put(210,40){\circle*{4}}
    \put(240,40){\circle*{4}}
    \put(150,40){\oval(200,16)}
    \put(250,20){$G[V_2,V_1]$}
    \put(20,35){$G[V_1]$}
    \put(150,70){\circle*{4}$u$}
    \put(150,70){\line(1,-1){30}}
    \put(150,70){\line(2,-1){60}}
    \put(150,70){\line(-1,-1){30}}
    \put(150,70){\line(-2,-1){60}}
    \put(150,70){\line(-3,-1){90}}
    \put(150,70){\line(1,-2){10}}
    \put(150,70){\line(-1,-2){10}}
    \put(60,40){\line(5,1){20}}
    \put(90,40){\line(1,0){30}}
    \put(90,40){\line(3,1){15}}
    \put(120,40){\line(1,0){10}}
    \put(120,40){\line(2,-1){10}}
    \put(120,40){\line(-2,1){10}}
    \put(180,40){\line(1,0){30}}
    \put(180,40){\line(-2,1){10}}
    \put(180,40){\line(-2,-1){10}}
    \put(210,40){\line(1,0){10}}
    \put(210,40){\line(1,-2){5}}
    \put(210,40){\line(-2,-1){10}}
    \put(240,40){\line(-2,-1){10}}
    \qbezier(60,40)(90,30)(120,40)
   \put(60,10){\line(3,2){15}}
    \put(60,10){\line(0,1){30}}
    \put(90,10){\line(1,0){30}}
    \put(90,10){\line(1,1){10}}
    \put(90,10){\line(2,1){10}}
    \put(120,10){\line(1,2){5}}
    \put(120,10){\line(-2,1){10}}
    \put(180,10){\line(-2,1){10}}
    \put(180,10){\line(-1,0){10}}
    \put(180,10){\line(-1,2){5}}
    \put(180,10){\line(0,1){10}}
    \put(210,10){\line(1,2){5}}
    \put(210,10){\line(-2,1){10}}
    \put(240,10){\line(0,1){30}}
    \put(240,10){\line(-1,2){5}}
    \qbezier(180,10)(210,0)(240,10)
\end{picture}
    \caption{Label the red edges according to the partial vertex sums of vertices in $V_2\cup \{u\}$. If two vertices in $V_1$ have the same vertex sum, then rerrange the labels of the edges incident to $u$.}
    \label{fig:radius2}
\end{figure}

In particular, we have the following result for Conjecture~\ref{HMS10}

\begin{corollary}\label{radius2}
Every graph of radius two admits an antimagic orientation. 
 \end{corollary}

A well known theorem in random graph theory is that almost every graph has diameter two~\cite{MM66}. For the sake of completeness, we present its proof here. Recall that the Markov's inequality says that if $X$ is a random variable with nonnegative values, then $P(X\ge t)\le E(X)/t$ for any $t>0$.

\begin{theorem}\label{diameter}{\rm \cite{MM66}}
    Almost every graph has diameter two.
\end{theorem}

\begin{proof} 
Let $X(G(n,\frac{1}{2}))$ denote the number of pairs of vertices in $G(n,\frac{1}{2})$ which has no common neighbor. Thus, $X=0$ implies that $G(n,\frac{1}{2})$ has diameter two. Set $t=1$ in the Markov's inequality, and we show $E(X)\rightarrow 0$. 
For any two unordered vertices $v_i$ and $v_j$, define the random variable $X_{ij}$ to be $1$ if $v_i$ and $v_j$ have no common neighbor and $0$ otherwise. Then $P(X_{ij}=1)=[1-(\frac{1}{2})^2]^{(n-2)}$ and $E(X)=\sum_{1\le i<j\le n}E(X_{ij})=\binom{n}{2}(\frac{3}{4})^{n-2}$. Therefore, $E(X)\rightarrow 0$ as $n\rightarrow\infty$.    
\end{proof}

Since every graph of diameter two has a radius at most two, combining Corollary~\ref{radius2} and Theorem~\ref{diameter}, we obtain Theorem~\ref{random}.

\end{document}